\documentclass[12pt]{amsart}
\usepackage{amsmath}
\usepackage{amsfonts,amssymb,mathrsfs,amscd}
\usepackage[matrix,arrow,curve]{xy}

\numberwithin{equation}{section}

\theoremstyle{plain}
\newtheorem{theorem}{Theorem}[section]
\newtheorem{theoremintro}{Theorem}
\newtheorem{lemma}[theorem]{Lemma}
\newtheorem{prop}[theorem]{Proposition}
\newtheorem{corollary}[theorem]{Corollary}

\theoremstyle{definition}
\newtheorem{definition}[theorem]{Definition}
\newtheorem{remark}[theorem]{Remark}
\newtheorem{example}[theorem]{Example}

\def\m{\mathfrak m}
\def\p{\mathfrak p}
\def\Z{\mathbb Z}
\def\AAA{\mathcal A}
\def\SSS{\mathcal S}
\def\DD{\mathcal D}
\def\II{\mathcal I}
\def\TT{\mathcal T}
\def\OO{\mathcal O}
\def\kk{\mathsf k}
\def\MMod{\mathrm{{-}Mod}}
\def\mmod{\mathrm{{-}mod}}
\def\xra{\xrightarrow}
\def\bul{\bullet}
\renewcommand{\le}{\leqslant}
\renewcommand{\ge}{\geqslant}

\DeclareMathOperator{\Hom}{\textup{Hom}}
\DeclareMathOperator{\Ext}{\textup{Ext}}
\DeclareMathOperator{\End}{\mathrm{End}}
\DeclareMathOperator{\Spec}{\mathrm{Spec}}
\DeclareMathOperator{\coker}{\mathrm{coker}} 

\DeclareMathOperator{\hocolim}{\mathrm{hocolim}}
\DeclareMathOperator{\Supp}{\mathrm{Supp}}
\DeclareMathOperator{\coh}{\mathrm{coh}}
\DeclareMathOperator{\qcoh}{\mathrm{qcoh}}
\DeclareMathOperator{\Perf}{\mathrm{Perf}}
\DeclareMathOperator{\Rdim}{\mathrm{Rdim}}
\DeclareMathOperator{\gldim}{\mathrm{gldim}}
\DeclareMathOperator{\ann}{\mathrm{ann}}

\begin{document}

\title{Regular subcategories in bounded derived categories of affine schemes}
\author{Alexey Elagin}
\address{Institute for Information Transmission Problems (Kharkevich Institute), Moscow, RUSSIA\\
National Research University Higher School of Economics, Russian Federation}
\email{alexelagin@rambler.ru}
\author{Valery Lunts}
\address{Indiana University, Bloomington, USA\\
National Research University Higher School of Economics, Russian Federation}
\email{vlunts@indiana.edu}



\begin{abstract}
Let $R$ be a commutative Noetherian ring such that $X=\Spec R$ is connected. We prove that the category $D^b(\coh X)$ contains no proper full triangulated subcategories which are strongly generated. We also bound from below the Rouquier dimension of a triangulated category $\TT$, if there exists a triangulated functor $\TT \to D^b(\coh X)$ with certain properties. Applications are given to cohomological annihilator of $R$ and to point-like objects in $\TT$.

\end{abstract}

\maketitle



\footnotetext[0]{The study of regular subcategories on affine schemes (Chapters 2--3) was carried out by A.\,Elagin at the IITP RAS at the expense of the Russian
Foundation for Sciences (project  14-50-00150). The study of f-approximations and applications to point-like objects (Chapters 4--5) was performed by V.\,Lunts and partially supported by Laboratory of Mirror Symmetry NRU HSE, RF Government grant, ag. № 14.641.31.0001.}

\section{Introduction} 
In \cite{BVdB} Bondal and Van den Bergh proved the following result. Let $\TT$ be a $\kk$-linear 
$\Ext$-finite Karoubian triangulated category over a field $\kk$. Suppose $\TT$   has a strong generator (we recall the definition in Section 2.6). Then any cohomological functor $\TT\to\kk\mmod$ of finite type is representable. This theorem has nice consequences: in particular, any fully faithful functor from $\TT$ to a  $\kk$-linear 
$\Ext$-finite triangulated category is an embedding of a semi-orthogonal component.
Later Rouquier \cite{Rou} generalized the above result for categories linear over a Noetherian commutative ring.

Following D.\,Orlov, we call triangulated categories with a strong generator \emph{regular}.

Algebraic geometry provides many examples of regular triangulated categories: for example, such are the bounded derived categories of separated schemes of finite type over a perfect field (see \cite{Rou} or \cite{BVdB} for the case of 
smooth projective varieties over a field). See Example~\ref{example_example} for other examples of regular categories.

For a smooth projective variety $X$ the category $D^b(\coh X)$ can have regular triangulated subcategories: indeed, any semi-orthogonal component of a regular triangulated category is regular. In particular, a subcategory generated by an exceptional collection is regular.

It turns out that the situation with affine varieties is completely different: there are no non-trivial regular subcategories in $D^b(\coh X)$ for affine $X$. To be more precise, we prove the following

\begin{theorem} 
\label{theorem_1}{\rm (=Theorem  \ref{theorem_regular})}
Let $R$ be a commutative Noetherian ring.
Assume that $X:=\Spec R$ is connected. Let $\TT\subset D^b(\coh X)$ be a thick triangulated subcategory which is regular. Then
$\TT=0$ or $\TT=D^b(\coh X)$.
\end{theorem}

Note that we do not assume that $R$ is an algebra over a field or finitely generated over some other ring.

Theorem \ref{theorem_1} is a generalization (in the commutative case) of the following result by S.\,Oppermann and J.\,Stovicek \cite[Theorem 2]{OS}:
\begin{theoremintro}
Let $\Lambda$ be  an algebra, finite as a module over some commutative Noetherian ring. Let  $\TT\subset D^b(\mathrm{mod}-\Lambda)$ be a thick subcategory   such that $\Perf(\Lambda)\subset \TT$. If $\TT$ is regular then $\TT=D^b(\mathrm{mod}-\Lambda)$. 
\end{theoremintro} 

The difference with our result is that we do not need the $\Perf(\Lambda)\subset \TT$ assumption. In particular, our Theorem \ref{theorem_1} is applicable for smooth affine schemes (where $\Perf(\Lambda)=D^b(\mathrm{mod}-\Lambda)$).

Our proof of Theorem \ref{theorem_1} is quite different from the proof in \cite{OS}, it uses a variant of the cited above representability result (mainly following \cite{BVdB} and \cite{Rou}). As an immediate consequence we obtain the following 

\begin{theorem} {\rm (=Theorem \ref{theorem_appl})}
Let $X\neq \emptyset$ and $Y$ be two Noetherian affine schemes. Assume that $Y$ is connected and that the category $D^b(\coh X)$ is regular.
Then any fully faithful functor $D^b(\coh X)\to D^b(\coh Y)$ is an equivalence.
\end{theorem}

In the second part of the paper we consider a more general situation: a regular triangulated category $\TT$ and a functor (not necessarily full or faithful) $\TT\to D^b(R\mmod)$. We prove that if the image of this functor contains a collection of objects with certain properties, then the Rouquier dimension of $\TT$ cannot be smaller than a certain number. Here is an example of such situation:

\begin{theorem}
\label{theorem_2}
{\rm (=Corollary \ref{cor_bounddim})}
Let $R$ be a commutative Noetherian integral domain. Let $\TT$ be a regular triangulated category with Rouquier dimension $n$, and $F\colon \TT \to D^b(R\mmod)$ be a triangulated functor such that the essential image of $F$ contains a family $\{M_i\}_{i\in I}$ of finitely generated $R$-modules. Suppose there exist an integer $d\ge 1$ and a family of $R$-modules $\{N_i\}_{i\in I}$ such that $\cap_i \ann\Ext^d(M_i,N_i)=0$.  Then $d\le n$.
\end{theorem}

The proof of Theorem~\ref{theorem_2} and of its partner Theorem \ref{theorem_main} is based on the notion of an $f$-approximation, see Definition~\ref{definition_fapprox}. Roughly it describes the situation when one object can be obtained by some number of steps from another object ``up to an object'' which is annihilated by an element $f\in R$.

An application of Theorem~\ref{theorem_2} is given in Section \ref{section_pointlike} where we study point-like objects in regular categories. Recall that an object $E$ of an enhanced  triangulated $\kk$-linear category $\TT$ is called \emph{point-like} of dimension $r$ if 
the dg-algebra $\End(E)$ is quasi-isomorphic to the exterior algebra $\Lambda(\kk^r)$  (where $\deg(\kk^r)=1$ and the differential is zero). For instance, such is the structure sheaf $\OO_x$ of a regular point $x$ with the residue field $\kk$ on a $\kk$-scheme $X$, where $r$ is the dimension of $X$ at $x$.
We prove

\begin{theorem} {\rm (=Theorem \ref{theorem_point})}
Let $\TT$ be an enhanced $\kk$-linear $\Ext$-finite triangulated \nobreak{category}. Suppose $\TT$ has finite Rouquier dimension $n$ and contains a point-like object $E$ of dimension~$r$. 
Then $r\le n$.
\end{theorem}
 
Another application of our technique gives an upper bound, in terms of Rouquier dimension of $D^b(R\mmod)$, on the minimal number $d$ such that the $d$-th  cohomological annihilator of a Noetherian ring $R$ is non-trivial, 
 see  Theorem~\ref{theorem_ca}.
 
The paper is organized as follows. In Section~\ref{section_prelim} we recall the basic definitions and  background material. What is the most interesting here is Lemma~\ref{lemma_supp} which is surely well-known  but for which we do not know any reference. 
In Section~\ref{section_regsubcat} we formulate and prove Theorem~\ref{theorem_1}.
In Section~\ref{section_f} we introduce and develop the technique of $f$-approximations, and use it for proving Theorems~\ref{theorem_2} and \ref{theorem_main}. Applications to point-like objects occupy Section~\ref{section_pointlike}.
 
We thank Amnon Neeman and Michel Van den Bergh for useful conversations. We are grateful to Osamu Iyama who made us aware of the results by Oppermann and Stovicek.
Also we are grateful to the anonymous referee for the careful reading of the manuscript and for numerous valuable remarks.

\section{Notation, definitions and background material}
\label{section_prelim}

\subsection{Categories}
Throughout this paper $R$ will denote a commutative Noetherian ring and $X=\Spec R$ will denote the spectrum of $R$. By $R\MMod$ we will denote the abelian category of $R$-modules, while $R\mmod\subset R\MMod$ would be used for the abelian subcategory of finitely generated $R$-modules.
Let $D(R\MMod)$ be the unbounded derived category of $R$-modules and $D^b(R\mmod)$ be the bounded derived category of finitely generated $R$-modules.
Often it is convenient to work with the category $D^b_{fg}(R)$ of complexes in $D(R\MMod)$
which have finitely many cohomology modules and whose cohomology modules are finitely generated. This category $D^b_{fg}(R)$ is a strictly full subcategory in $D(R\MMod)$ and is equivalent to $D^b(R\mmod)$. We also have the following full subcategories $D^-_{fg}(R), D^+_{fg}(R), D_{fg}(R)\subset D(R\MMod)$ (with the obvious meaning). The category $\Perf(R)$ is the full subcategory of $D(R\MMod)$ consisting of complexes which are quasi-isomorphic to finite complexes consisting of projective modules of finite rank.
The category $D^-_{fg}(R)$ is equivalent to $D^-(R\mmod)$.

Sometimes we will use the geometric language. Recall that $R\MMod$ is equivalent to $\qcoh X$, the category of quasi-coherent sheaves on $X$, and $R\mmod$ is equivalent to $\coh X$, the category of coherent sheaves on $X$. Hence, $D(R\MMod)$ is equivalent to $D(\qcoh X)$, the unbounded derived category of quasi-coherent sheaves on $X$, and $D^b(R\mmod)$ is equivalent to $D^b(\coh X)$, the bounded derived category of coherent sheaves on $X$. Under these equivalences $D^b_{fg}(R)$ corresponds to $D^b_{\coh}(X)$, the strictly full  subcategory of $D(\qcoh X)$ formed by complexes whose cohomology sheaves are coherent and vanish except for a finite number.  We also have the geometric counterparts $D^-_{\coh}(X)$, $D^+_{\coh}(X)$, $D_{\coh}(X)$, $\Perf(X)$ of the categories $D^-_{fg}(R)$, $D^+_{fg}(R)$, $D_{fg}(R)$, $\Perf(R)$. The category $D^b_{\coh}(X)$ is equivalent to $D^b(\coh X)$ and is a good substitute for the latter, similarly $D^-_{\coh}(X)\simeq D^-(\coh X)$.

For objects $M,N$ of a triangulated category $\TT$ we use the following standard notation:
$$\Hom^i(M,N):=\Hom(M,N[i]).$$ 
By $\Hom^{\bul}(M,N)$ we denote the graded abelian group $\oplus_i\Hom^i(M,N)[-i]$, unless it is stated explicitly otherwise.

For $a\in \Z$ we denote by $D^{\ge a}_{\coh}(X)$ (resp. $D^{\le a}_{\coh}(X)$)  the full subcategory in $D_{\coh}(X)$ consisting of objects $C$ such that $H^i(C)=0$ for any $i<a$ (resp. $i>a$). Similar notation is used for other variants of derived categories.

\subsection{Dualizing complex}
\label{section_dualizing}
Recall that a \emph{dualizing complex} on $X$ is an object $\DD \in D^b(\coh X)$ such that the functor $R\mathcal Hom(-,\DD)$ induces an involutive anti-autoequivalence on the category $D^b(\coh X)$. It is known that if $\DD$ exists it is unique up to tensoring with a line bundle and a shift, see \cite[V.3]{Ha}. It is also known that the same functor induces an involutive anti-autoequivalence of the category $D_{\coh} (X)$ and  defines an anti-equivalence between the categories $D^-_{\coh}(X)$ and $D^+_{\coh} (X)$. The dualizing complex is known to exist for example if $R$ is a complete local Noetherian ring, see \cite[V.10]{Ha}.

\subsection{Orthogonality}

Let $M$ be an $R$-module.  Recall that the \emph{support}
$\Supp M$ of $M$ is defined to be the subset of $X=\Spec R$ consisting of all points $x\in X$ (i.e. prime ideals $\p\subset R$)
such that the localization $M_x:=M_{\p}=M\otimes_R R_{\p}$ is not equal to $0$. If the module~$M$ is finitely generated then $\Supp M$ is closed in $X$ in Zariski topology. Also, for a finitely generated module $M$ one has $\Supp M=\emptyset$ if and only if $M=0$. For a complex $M\in D(R\MMod)$ the support $\Supp M$ is defined as
$$\cup_i \Supp H^i(M).$$
Equivalently, $\Supp M$ consists of point $x\in X$  such that the localization $M_{x}$ is a non-zero object of $D(R_{\p}\MMod)$. If $M\in D^b(R\mmod)$ then $\Supp M$ is closed in $X$ and $\Supp M=\emptyset$ if and only if $M=0$.

\begin{lemma}
\label{lemma_otimes}
Let $M,N\in D^-_{fg}(R)$.
\begin{enumerate}
\item Suppose $R$ is local. If $M,N\ne 0$ then $M\otimes_R^L N\ne 0$.
\item In general, suppose  $\Supp M\cap \Supp N\ne \emptyset$. Then $M\otimes^L_R N\ne 0$.
\end{enumerate}
\end{lemma}
\begin{proof}
1. Without loss of generality we can assume that
$$0=\max\{i\colon H^i(M)\ne 0\}=\max\{i\colon H^i(N)\ne 0\}.$$
Then $H^0(M\otimes^L_RN)=H^0(M)\otimes_R H^0(N)$. The statement follows from the fact that
the tensor product of two nonzero finitely generated $R$-modules is nonzero, see \cite[Ex.~2.3]{AM}.

2. Take $x\in \Supp M\cap \Supp N$. By definition, $M_{x}$ and $N_{x}$ are nonzero.
Note that
$$(M\otimes^L_RN)_{x}\simeq M_{x}\otimes^L_{R_{x}}N_{x}.$$
By part 1, we have  $M_{x}\otimes^L_{R_{x}}N_{x}\ne 0$, hence $M\otimes^L_RN\ne 0$.
\end{proof}

\begin{lemma}
\label{lemma_supp}
Let $M\in D^-_{fg}(R\mmod), N\in D^+_{fg}(R\mmod)$ be complexes. Then $\Hom^{\bul}(M,N)=0$ if and only if $\Supp M\cap \Supp N= \emptyset$.
\end{lemma}
\begin{proof} The ``if'' implication is clear, we shall prove the ``only if'' statement.

First, we reduce to the case when $R$ is a complete local Noetherian ring. 
Let $x\in \Supp M\cap \Supp N$. Consider the corresponding local ring $R_x$ and its completion
$\hat{R}_x$. The functor $(-)\otimes _R\hat{R}_x$ is exact since $\hat{R}_x$ is a flat $R$-algebra. Recall that $M$ is quasi-isomorphic to an object of $D^-(R\mmod)$. For
$R$-modules $K,L$, where $K$ is finitely generated,  we have
$$\Hom _R(K,L)\otimes _R\hat{R}_x=\Hom _{\hat{R}_x}(\hat{K}_x,\hat{L}_x)$$
and hence 
$$\Hom ^\bullet _R(M,N)\otimes _R\hat{R}_x=\Hom^\bullet _{\hat{R}_x}(\hat{M}_x,\hat{N}_x).$$
(Here we denote by $\hat{K}_x$ the $\hat{R}_x$-module $K\otimes _R\hat{R}_x$ and similarly for $L$, $M$ and $N$). Moreover,
if $x\in \Supp M \cap \Supp N$, then also $x\in \Supp \hat{M}_x \cap \Supp \hat{N}_x$. Thus we may
assume that $R$ is a complete local Noetherian ring. Then $R$ admits a dualizing complex, see Section~\ref{section_dualizing}. 

Let $\DD\in D^b_{fg}(R)$ be a dualizing complex.
Denote $N^{\vee}=R\mathcal Hom(N,\DD)$. Then $N\simeq R\mathcal Hom(N^{\vee},\DD)$ and
$$
\Hom^i(M,N)\simeq \Hom^i(M,R\mathcal Hom(N^{\vee},\DD))\simeq 
\Hom^i(M\otimes^L_R N^{\vee},\DD).
$$
Since duality is an anti-equivalence, it suffices to check that $M\otimes^L N^{\vee}\ne 0$.
Note that $N^{\vee}\in D^-_{fg}(R)$ and use Lemma~\ref{lemma_otimes} to complete the proof.
\end{proof}

\subsection{Semi-orthogonal decompositions} 

For a triangulated subcategory (or just for a set of objects) $\SSS\subset \TT$ in a triangulated category $\TT$ its \emph{right orthogonal} $\SSS^{\perp}$ (resp. \emph{left orthogonal} $^{\perp}\SSS$ ) is defined as
the full subcategory in $\TT$ consisting of all objects $F$ such that $\Hom^i(S,F)=0$ (resp. $\Hom^i(F,S)=0$) for all $i\in\Z$ and $S\in\SSS$. A \emph{semi-orthogonal decomposition} $\TT =\langle \TT _1,\TT_2\rangle$ of a triangulated category~$\TT$ is a pair $\TT _1 ,\TT _2\subset \TT$ of full triangulated subcategories such that $\TT _1\subset \TT _2^\perp$, and each object $A\in \TT$ appears in a distinguished triangle $A_2 \to A \to A_1\to A_2[1]$, where $A_i\in \TT _i$.

Lemma \ref{lemma_supp} has the following neat corollaries.

\begin{corollary} 
\label{cor_support}
Let $R$ be a commutative Noetherian ring and $X=\Spec R$. 
Let $M\in D^-_{fg}(R)$  be a complex such that $\Supp M=X$. Then 
$$M^{\perp}\cap D^+_{fg}(R)=0.$$
Similarly, let  $N\in D^+_{fg}(R)$ be a complex such that $\Supp N=X$.
Then 
$$^{\perp}N\cap D^-_{fg}(R)=0.$$
\end{corollary}

\begin{corollary} 
\label{cor_semi}
Let $R$ be a commutative Noetherian ring and $X=\Spec R$. Suppose that $D^b(\coh X)=\langle \TT _1 ,\TT _2\rangle $ is a semi-orthogonal decomposition. Then $D^b (\coh X)=\TT _1 \oplus \TT _2$, i.e. the semi-orthogonal decomposition is an orthogonal one. Moreover, if $X$ is connected then $\TT _1 =0$ or $\TT _2 =0$.

Similar statements hold for $\Perf(R)$ in place of $D^b(\coh X)$.
\end{corollary}

\begin{proof} Lemma \ref{lemma_supp} implies that for any $A_1\in \TT _1$ and $A_2\in \TT _2$ we have $\Supp (A_1)\cap \Supp (A_2)=\emptyset$ (otherwise $A_1 \notin \TT _2^\perp$). Again by Lemma  \ref{lemma_supp} we have $\Hom^{\bul}(A_1,A_2)=0$ and thus $D^b(\coh X)=\TT _1\oplus \TT _2$. To prove the second assertion, assume the contrary: $\TT_1,\TT_2\ne 0$. Note that there exists a distinguished triangle $A_2\to \mathcal O _X \to A_1\to A_2[1]$ in $D^b(\coh X)$ where $A_i\in \TT _i$. Thus $X= \Supp \OO_ X =\Supp A_1 \cup \Supp A_2$. Suppose that $A_1=0$. Then $\mathcal O _X=A_2 \in \TT _2$. Hence $\TT_1\subset \OO_X^{\perp}=0$ by Corollary~\ref{cor_support}, a contradiction. Similarly we check that $A_2\ne 0$. Therefore $A_1,A_2\neq 0$, and then $X$ is the union of the nonempty closed subsets $\Supp (A_1)$ and $\Supp (A_2)$. By orthogonality of $A_1$ and $A_2$ and Lemma  \ref{lemma_supp}, one has $\Supp (A_1)\cap\Supp (A_2)=\emptyset$, hence $X$ is disconnected. A contradiction.
\end{proof}

\subsection{Karoubian categories}

A full triangulated subcategory $\TT'\subset \TT$ of a triangulated category is called \emph{thick} if any object in $\TT$ which is a direct summand of an object in $\TT'$ belongs to $\TT'$. Note that this property depends not only on $\TT'$ but also on $\TT$.

For a full triangulated subcategory $\TT'$ of a triangulated category $\TT$ one can consider its \emph{idempotent} or \emph{thick closure} in $\TT$: it is the full subcategory in $\TT$ formed by objects of $\TT$ which are isomorphic to direct summands of objects of $\TT'$. The thick closure of a triangulated subcategory is also triangulated.

A triangulated category $\TT$ is called {\it Karoubian} if every idempotent splits in $\TT$.
The following proposition is well known.

\begin{prop} 
\label{prop_karoubian} 
Let $R$ be a Noetherian ring and $X=\Spec R$.
\begin{enumerate}
\item The category $D(\qcoh X)$ is Karoubian.

\item The category $D^b(\coh X)$ is Karoubian.

\item A thick subcategory of $D^b(\coh X)$ is Karoubian.

\item The category $\Perf (X)$ is Karoubian.
\end{enumerate}
\end{prop}

\begin{proof} 1. It is a classical fact (see \cite[Prop.3.2]{BN}) that a triangulated category with countable direct sums is Karoubian.

2. Since $D^b(\coh X)\simeq D^b_{\coh}(X)\subset D(\qcoh X)$, by part 1 a direct summand of an object in $D^b_{\coh}(X)$ exists in $D(\qcoh X)$.
But then it clearly is in $D^b_{\coh}(X)$.

3. This follows from part 2 because any thick subcategory of a Karoubian category is Karoubian.

4. This is proved in \cite[Prop. 3.4]{BN}.
\end{proof}

\subsection{Regular categories}
\label{section_regularcategories}
Let $G$ be an object of a triangulated category $\TT$. By $[G]_n$ we denote the full subcategory in $\TT$ consisting of objects, generated by $G$ in $n$ steps. To be more precise, let $[G]_0$ be the full subcategory whose objects are finite direct sums of shifts of $G$. For $k\ge 1$ the full subcategory $[G]_k$  is defined as the subcategory of objects $F$ such that there exists a distinguished triangle $F_{k-1}\to F\to F_{0}\to F_{k-1}[1]$ with $F_0\in [G]_0$ and $F_{k-1}\in [G]_{k-1}$.
By $\langle G \rangle_n$ we denote the idempotent closure of $[G]_n$ in $\TT$.

Note here that our usage of notation $\langle G\rangle_n$ is different from the one in \cite{BVdB} or \cite{Rou}. We count the number of cones used to construct an object, not the number of building blocks.

The next lemma is standard.
\begin{lemma} \label{lemma_wellknown1}
\begin{enumerate}
\item If $A\in [G]_k$, $B\in [G]_l$ and $g\colon A\to B$ is a morphism, then $Cone(g)$ lies in $[G]_{k+l+1}$. The subcategory $\cup_n [G]_n$ is triangulated in $\TT$.
\item If $A\in \langle G \rangle_k$, $B\in \langle G \rangle_l$ and $g\colon A\to B$ is a morphism, then  $Cone(g)$  lies in $\langle G \rangle_{k+l+1}$. The subcategory $\cup_n \langle G\rangle_n$ is triangulated and thick in $\TT$.
\end{enumerate}
\end{lemma}

\begin{proof} Part 1 follows from the ``associativity lemma'' \cite[Lemma 1.3.10]{BBD} and part 2 is a consequence of \cite[Lemma 2.2.1]{BVdB}.
\end{proof}

The category $[G]=\cup_n [G]_n$ is the smallest full triangulated subcategory in $\TT$ containing~$G$.
The category $\langle G \rangle=\cup_n \langle G \rangle_n$ is said to be \emph{classically generated by $G$}. It means that it is the smallest full triangulated thick subcategory in $\TT$ containing~$G$. 
Also note that in general the categories $[G]$ and $\langle G\rangle$ are not closed under (infinite) direct sums in~$\TT$.

Further, one defines analogous subcategories involving taking all possible direct sums. Note that we do not suppose that $\TT$ has all direct sums or direct sums of some fixed cardinality. We consider only direct sums that exist in $\TT$.

Let $[G]^{\oplus,\TT}_0$ be the full subcategory of $\TT$ whose objects are all direct sums of shifts of~$G$ that exist in $\TT$. For $k\ge 1$ the full subcategory $[G]^{\oplus,\TT}_k$  is defined as the subcategory consisting of objects $F$ such that there exists a distinguished triangle $F_{k-1}\to F\to F_{0}\to F_{k-1}[1]$ with $F_0\in [G]^{\oplus,\TT}_0$ and $F_{k-1}\in [G]^{\oplus,\TT}_{k-1}$.
By $\langle G \rangle^{\oplus,\TT}_n$ we denote the idempotent closure of $[G]^{\oplus,\TT}_n$ in $\TT$. Also we put
$[G]^{\oplus,\TT}=\cup_n [G]^{\oplus,\TT}_n$  and $\langle G \rangle^{\oplus,\TT}=\cup_n\langle G \rangle^{\oplus,\TT}_n$.

The next lemma is also well known.
\begin{lemma}
\label{lemma_generation}

\begin{enumerate}
\item If $A\in [G]^{\oplus,\TT}_k$, $B\in [G]^{\oplus,\TT}_l$ and $g\colon A\to B$ is a morphism, then  $Cone(g)$  lies in $[G]^{\oplus,\TT}_{k+l+1}$. The subcategory $\cup_n [G]_n^{\oplus,\TT}$ is triangulated in~$\TT$.
\item If $A\in \langle G \rangle^{\oplus,\TT}_k$, $B\in \langle G \rangle^{\oplus,\TT}_l$ and $g\colon A\to B$ is a morphism, then  $Cone(g)$ lies in $\langle G \rangle^{\oplus,\TT}_{k+l+1}$. The subcategory $\cup_n \langle G\rangle_n^{\oplus,\TT}$ is triangulated and thick in $\TT$.
\item The subcategories $[G]^{\oplus,\TT}_n$ and $\langle G\rangle^{\oplus,\TT}_n$ in $\TT$ are closed under forming all direct sums in $\TT$.
\end{enumerate}
\end{lemma}

\begin{proof} Parts 1 and 2 are proved exactly as parts 1 and 2 in Lemma \ref{lemma_wellknown1} and part 3 follows from the fact that an arbitrary sum of distinguished triangles is a distinguished triangle, see \cite[Remark 1.2.2]{Ne}.
\end{proof}

\begin{definition}

An object $G$ of a triangulated category $\TT$ is said to be a \emph{classical generator} if $\TT=\langle G\rangle$. An object $G\in\TT$ is said to be a \emph{strong generator} if $\TT=\langle G\rangle_n$ for some~$n$. A triangulated category $\TT$ is called \emph{regular} if it has a strong generator.

The \emph{Rouquier dimension} $\Rdim{\TT}$ of a regular triangulated category $\TT$ is defined as the minimal integer $n$ such that $\TT =\langle G\rangle _n$ for some $G\in \TT$.
\end{definition}

Note that the regularity of quasi-projective algebraic varieties  is consistent with the regularity of their categories of perfect complexes, see Example~\ref{example_example}(5) below.

\begin{remark}[({See \cite[Chapter 3]{Rou}})]
\label{remark_allstrong}
Any generator of a regular triangulated category is a strong one. Indeed, suppose $\TT=\langle G\rangle_n$ and $G'$ is a generator of $\TT$. Then $G\in \langle G'\rangle_k$ for some $k$. It follows from Lemma~\ref{lemma_wellknown1} that $\TT=\langle G\rangle_n\subset \langle G'\rangle_{(n+1)(k+1)-1}$, therefore $G'$ is also a strong generator of $\TT$.
\end{remark}

\begin{example}
\label{example_example}
\begin{enumerate}
\item Let $R$ be a commutative ring that is either essentially of finite type over a field or an equicharacteristic excellent local ring. Then $D^b(R\mmod)$ is regular. \cite[Th.~1.4 or Cor. 7.2]{IT2}
\item Let $R$ be a Noetherian commutative ring. Then $\Perf(R)$ is regular if and only if $R$ has finite global dimension. 
\item Let $Y$ be a smooth scheme over a field. Then $D^b(\coh Y)$ is regular. \cite[Th. 3.1.4]{BVdB}
\item Let $Y$ be a separated scheme of finite type over a perfect field. Then $D^b(\coh Y)$ is regular. \cite[Th. 7.38]{Rou}
\item Let $Y$ be a quasi-projective scheme over a field. Then $\Perf(Y)$ is regular if and only if $Y$ is regular. \cite[Prop. 7.34]{Rou}
\end{enumerate}
\end{example}

There exists a commutative Noetherian regular ring $R$ for which the category $D^b(R\mmod)$ is equivalent to $\Perf(R)$ and is not regular, see Example~\ref{example_nagata} below.

\subsection{Compact objects}
\begin{definition}
Let $\TT$ be an additive category. An object $M\in \TT$ is called \emph{compact} if $\Hom(M,-)$ commutes with direct sums  that exist in $\TT$. That is,
the natural map
$$\oplus_{\alpha}\Hom(M,M_{\alpha}) \to \Hom(M,\oplus_{\alpha}M_{\alpha})$$
is an isomorphism for any $\oplus_{\alpha}M_{\alpha}$ that exists in $\TT$. We denote by $\TT ^c$ the full subcategory of compact objects in $\TT$. If $\TT$ is triangulated then $\TT^c$ is also triangulated.
\end{definition}

The following facts are well known.

\begin{lemma}
\label{lemma_neeman}
Let $\TT$ be a triangulated category. Let $G$ be a compact object in $\TT$. 
Then
$$\langle G\rangle^{\oplus,\TT}\cap \TT^c=\langle G\rangle.$$
\end{lemma}

\begin{proof}  See \cite[Prop. 2.2.4]{BVdB} or \cite[Cor. 3.14]{Rou} for the case when $\TT$ is not required to have arbitrary direct sums.
\end{proof}

\begin{lemma}
\label{lemma_perfcompact}
Let $R$ be a commutative ring. Then $D(R\MMod)^c=\Perf(R)$.
\end{lemma}
\begin{proof}
See \cite[Prop. 6.4]{BN}.
\end{proof}

\begin{lemma}
\label{lemma-d^b-compact} 
Let $X$ be an affine Noetherian scheme. 
Then we have $D^b(\coh X)\subset D^+(\qcoh X)^c$.
\end{lemma}

\begin{proof} Let $\{ F_i\}\subset D^+(\qcoh X)$ be a set of objects such that $\oplus _i F_i\in D^+(\qcoh X)$. Then for some $a\in\Z$ one has $\oplus_i F_i \in D^{\ge a}(\qcoh X)$. It follows that  $F_i \in D^{\ge a}(\qcoh X)$ for any $i$.

Let $E\in D^b(\coh X)$. Choose a perfect complex and a morphism $P\to E$ 
such that $E':=Cone(P\to E)\in D^{\le a-2}(\qcoh X)$. Indeed, one can take $P=\sigma^{\ge a-2}(Q)$ to be the stupid truncation  of a resolution $Q$ of $E$ by free finitely generated $\OO_X$-modules. Then $\Hom(E'[k],F_i)=\Hom(E'[k],\oplus_iF_i)=0$ for $k=-1,0$. Therefore in the commutative square
$$\xymatrix{\oplus_i \Hom(E,F_i) \ar[rr] \ar[d] && \Hom(E,\oplus_i F_i)\ar[d]\\
\oplus_i \Hom(P,F_i)\ar[rr] &&  \Hom(P,\oplus_i F_i) }$$
the vertical arrows are isomorphisms. By Lemma~\ref{lemma_perfcompact} $P$ is compact (in $D(\qcoh X)$ and thus in $D^+(\qcoh X)$), so the lower arrow is an isomorphism. It follows that the upper arrow is also an isomorphism. 
\end{proof}

\section{Regular subcategories of $D^b(\coh X)$}
\label{section_regsubcat}

Recall that $R$ denotes a commutative Noetherian ring and $X=\Spec R$.
We call $R$ \emph{indecomposable} if $R$ is not isomorphic to a direct product of two other rings; this is equivalent to $X$ being connected.

In this section we will prove the following

\begin{theorem}
\label{theorem_regular}
Let $R$ be an indecomposable commutative Noetherian ring. Let $\TT\subset D^b(R\mmod)$ be a thick subcategory which is regular. Then
$\TT=0$ or $\TT=D^b(R\mmod)$ (the second case is possible only if $D^b(R\mmod)$ is regular).
\end{theorem}

\begin{remark} In general, the category $D^b(R\mmod)$ may not be regular for a commutative Noetherian ring $R$, see Example \ref{example_nagata} below. We do not know if $D^b(R\mmod)$ has a classical generator for any commutative Noetherian ring $R$.
\end{remark}

Theorem \ref{theorem_regular} has the following companion which is much easier to prove.
We prefer to start with two proofs of this special case. 
The second given proof of Theorem \ref{theorem_regular0} uses the same methods as the proof of Theorem \ref{theorem_regular} but can be performed without technical difficulties.
Note that  Theorem \ref{theorem_regular0} is a consequence of Theorem \ref{theorem_regular} (since any subcategory in $\Perf (R)$ is also a subcategory in $D^b(R\mmod)$) and  
will not  be used in the proof of Theorem \ref{theorem_regular}.

\begin{theorem}
\label{theorem_regular0}
Let $R$ be an indecomposable commutative Noetherian ring. Let $\TT\subset \Perf(R)$ be a thick subcategory which is regular. Then
$\TT=0$ or $\TT=\Perf (R)$.
\end{theorem}
\begin{proof}
There is a classification of thick subcategories in $\Perf(R)$ due to M.\,J.\,Hopkins and A.\,Neeman \cite[Theorem 1.5]{Ne}. 
Let $Z\subset \Spec R$ be a subset, closed under specialization. Denote 
$$\TT_Z=\{C\in \Perf(R)\mid \Supp C\subset Z\}.$$
Then the correspondence $Z\mapsto \TT_Z$ induces a bijection between subsets of $\Spec R$ closed under specialization and thick subcategories of $\Perf (R)$.

Suppose $\TT_Z$ is regular, i.e. $\TT_Z=\langle G\rangle_m$ for some $G\in \Perf (R)$. Then $\Supp G=W\subset Z$ is a closed subset of $\Spec R$ and also $\Supp \langle G\rangle _m=W$. Hence $Z=W$ and so $Z$ is closed.
If $Z=\Spec R$ or $Z=\emptyset$ then $\TT_Z=\Perf(R)$ or $\TT =0$ and the theorem is proved.

Suppose $Z\subset \Spec R$ is a proper nonempty closed subset.
Let $\II_Z\subset R$ denote the ideal of $Z\subset \Spec R$ with the reduced scheme structure. Since $\TT_Z$ is generated by $G$ in a finite number of steps, by Lemma~\ref{lemma_ABC} below all objects of $\TT_Z$ are annihilated by some fixed power of $\II_Z$, say by $\II _Z^n$. 
We claim that $\II_Z^{n}=\II_Z^k$ for any $k\ge n$. 

Indeed, take any $k\ge n$. We would like to consider a perfect complex $C_k$ supported on $Z$ whose annihilator is contained in $\II_Z^k$. In general, the natural candidate $R/\II_Z^k$ is not a perfect complex, so we argue as follows.
Suppose $\II_Z=(f_1,\ldots,f_r)$ for some $f_i\in R$. 
Let $C_k$ be the Koszul complex $C_k=\otimes_{i=1}^r [R[1]\xra{f_i^k} R]$. We have $H^0(C_k)=R/(f_1^k,\ldots,f_r^k)$. Clearly, $C_k$ is perfect and supported on $Z$, therefore $C_k\in \TT_Z$. It follows that $C_k$ is annihilated by $\II_Z^n$. In particular, $\II_Z^n\cdot H^0(C_k)=0$, thus $\II_Z^n\subset (f_1^k,\ldots,f_r^k)\subset \II_Z^k$ for any $k$.

Now choose an irreducible  component $Y$ of $\Spec R$ such that $Y\cap Z\ne \emptyset, Y$.
Such a component exists because $\Spec R$ is connected. Let $\II_Y\subset R$ be the prime ideal of $Y$. Consider the ring $R_Y=R/\II_Y$ and the ideal $\II=(\II_Z+ \II_Y)/\II_Y\subset R_Y$, corresponding to the scheme-theoretical intersection $Y\cap Z$.
Note that $R_Y$ is a Noetherian integral domain and $\II\ne 0, R_Y$. Since $\II _Z^n=\II _Z^{k}$ for any $k\ge n$, it follows that $\II ^n=\II ^{k}$ for any $k\ge n$. By Krull`s Intersection Theorem  (see~\cite[Th. 10.17]{AM}), for every $f\in \cap _{k=1}^\infty \II ^k=\II ^n$ there exists $g\in \II $ such that $f\cdot (1+g)=0$. Since $R_Y$ is an integral domain and $1+g\ne 0$, it follows that $f=0$. Therefore $\II ^n=0$ and hence  $\II =0$ because $R_Y$ is an integral domain. We get a contradiction.
\end{proof}

There is another proof of Theorem~\ref{theorem_regular0} which does not use the classification of thick subcategories in $\Perf(R)$. We sketch this proof below.
\begin{proof}[of Theorem~\ref{theorem_regular0}, another one]
Since $\TT\subset \Perf(R)$, for any $M\in \TT$ and $N\in \Perf(R)$ the $R$-module
$$\oplus_{i\in \Z}\Hom_{\Perf(R)}(M[i],N)$$
is finitely generated. Therefore $\Hom_{\Perf(R)}(-,N)$ is a locally finite cohomological functor from $\TT$ to $R\mmod$, see \cite[Section 4.1.4]{Rou}.  Since $\TT$ is regular, it follows from~\cite[Prop. 4.9 and Cor. 4.17]{Rou} that this functor is representable by an object of $\TT$. Now by standard arguments we conclude that the embedding functor $\TT\to \Perf(R)$ has a right adjoint functor $\Perf(R)\to\TT$, i.e. $\TT$ is a semi-orthogonal component in $\Perf(R)$.
By Corollary~\ref{cor_semi}, we have $\TT=0$ or $\TT=\Perf(R)$.
\end{proof}

We get the immediate

\begin{corollary} If $D^b(R\mmod)\simeq \Perf (R)$ then Theorem \ref{theorem_regular} holds for $R$.
\end{corollary}

The necessary condition for the inclusion of the categories 
\begin{equation}
\label{eq_equiv-of-cat}
\Perf (R) \xra{\sim} D^b(R\mmod)
\end{equation}
to be an equivalence
is the regularity of the commutative Noetherian ring $R$. (Recall that a commutative ring $R$ is \emph{regular} if all its local ring are regular.) We do not know if this is sufficient. If $R$ is regular and of finite Krull dimension $d$, then every finitely generated $R$-module has a projective resolution of length $\le d$, hence in particular (\ref{eq_equiv-of-cat}) holds. But finiteness of Krull dimension is not necessary for (\ref{eq_equiv-of-cat}) to hold as the next example shows. 
We learned the following argument from Michel Van den Bergh.

\begin{example}
\label{example_nagata}
There is an example, constructed by Nagata, of a regular Noetherian ring~$R$ of infinite Krull dimension, see \cite[Ex. 11.4]{AM}. Recall the construction. Consider the polynomial ring $K=\kk[x_{i,j}]_{1\le j\le i<\infty}$ over some field $\kk$. Denote $\p_i=(x_{i,1}, x_{i,2},\ldots,x_{i,i})$, it is a prime ideal in $K$. Let $S\subset K$ be the complement to $\cup_i\p_i$ and $R$ be the localization $K[S^{-1}]$.
Note that the localization $R_{\p_i}$ is a regular local ring of Krull dimension $i$. 
Thus the global dimension of $R_{\p_i}$ is $i$ and the 
global dimension of $R$ is infinite.
We claim that nevertheless any finitely generated $R$-module $M$ has a finite projective resolution. Indeed, any such $M$ is isomorphic to the extension of scalars $M\simeq R\otimes _{R_0}M_0$,  over a subring $R_0\subset R$ which is a localization of a polynomial ring in  finitely many variables. Hence $R_0$ is regular and of finite Krull dimension. Moreover, $R$ is flat over $R_0$.
Hence the claim follows. 

It follows that we have an equivalence
$$D^b(R\mmod)\simeq\Perf(R).$$
We remark that these two categories are not regular, though they have a classical generator~$R$. Indeed, suppose $\Perf(R)$ has a strong generator, then by Remark~\ref{remark_allstrong} $R$ is a strong generator of $\Perf(R)$. Assume $\Perf(R)=\langle R\rangle_n$. It follows then from Ghost Lemma (see Lemma~\ref{lemma_ghost}) that $\Ext_R^i(M,N)=0$ for all finitely generated $R$-modules $M,N$ and $i>n$. Therefore 
$\gldim(R)\le n$, a contradiction.
\end{example}

\medskip
Now we turn to the proof of Theorem \ref{theorem_regular} in the general case. 
Let $G\in \TT$ be a strong generator of $\TT$. In the course of the proof it will be convenient for us to use the following strictly full  subcategories of $D(\qcoh X)$, which are introduced in Section~\ref{section_regularcategories}:
\begin{align*}
\TT_k&=\langle G\rangle_k, &\TT&=\cup_k \TT_k=\langle G\rangle;\\
\TT^+_k&=\langle G\rangle^{\oplus, D^+_{\coh}( X)}_k, & \TT^+&=\cup_k \TT^+_k=\langle G\rangle^{\oplus, D^+_{\coh}( X)};\\
\bar\TT_k&=\langle G\rangle^{\oplus, D(\qcoh X)}_k, & \bar\TT&=\cup_k \bar\TT_k=\langle G\rangle^{\oplus, D(\qcoh X)}.
\end{align*}
We point out that $\TT_k,\TT\subset D^b_{\coh}(X)$ and 
$\TT^+_k,\TT^+\subset D^+_{\coh}(X)$.
Note that all these categories are Karoubian, because $D(\qcoh X)$ is such and the categories in question are
by definition idempotent closed in $D(\qcoh X)$.

Next proposition is the main ingredient in the proof of Theorem~\ref{theorem_regular}.
\begin{prop}
\label{prop_adjoint}
The inclusion functor $\TT^+\to D^+_{\coh}(X)$
has a right adjoint functor
$D^+_{\coh}(X)\to \TT^+$.
\end{prop}

This proposition follows from Proposition~\ref{prop_Brown} and Lemma~\ref{lemma_TT} below.
Essentially one needs to show that for any $M\in D^+_{\coh} (X)$ the cohomological functor $\Hom(-,M)$ is representable on $\TT^+$.

For any $M\in D^+_{\coh} (X)$ we will consider the corresponding  cohomological functor on the category $\bar{\TT}$:
$$h^M(-):=\Hom(-,M)\colon \bar{\TT}\to R\MMod$$
 
\begin{definition} 
\label{definition_adm}
We say that a contravariant functor $H\colon \bar{\TT}\to R\MMod$ is {\it admissible} if
\begin{enumerate}
\item $H$ takes direct sums to direct products;
\item $H$ sends objects in $\TT$ to $R\mmod$;
\item for every $N\in \TT$ one has $H(N[k])=0$ for $k\gg 0$.
\end{enumerate}
\end{definition}

\begin{lemma}
\label{lemma_kernelofadmissible}
\begin{enumerate}
\item Let $M\in D^+_{\coh} (X)$. Then the functor $h^M\colon \bar{\TT}\to R\MMod$ is admissible.
\item Let $\beta \colon H_1\to H_2$ be a morphism between admissible functors. Then 
the kernel $\ker \beta$ is an admissible functor.
\end{enumerate}
\end{lemma}
\begin{proof}
\begin{enumerate}
\item
 Property 1 from Definition~\ref{definition_adm} follows from the  definition of a direct sum. 
Property 2 follows from the fact that $\Ext^i(K,L)$ is a finitely generated $R$-module for any $i$ and finitely generated $R$-modules $K$ and $L$, see \cite[Lemma II.3.2 or Prop. II.3.3]{Ha}.
Property 3 is clear, see for example \cite[Prop. 6.2]{Rou}.

\item  Property 3 from Definition~\ref{definition_adm} is obvious, property 2 holds since $R$ is Noetherian. Property 1 follows from the fact that $\ker$ commutes with direct products on the category of $R$-modules.
\end{enumerate}
\end{proof}

\begin{prop}
\label{prop_Brown}
Let $H\colon \bar\TT \to R\MMod$ be an admissible  cohomological functor. Then for any $k$ the functor $H$ restricted to $\bar \TT_k$ is a direct summand of the representable functor $h^C$ for some  $C\in \TT^+$.
\end{prop}

The proof is similar to the arguments in \cite{BVdB}. Essentially, we prove once again Brown representability theorem. The difference with \cite{BVdB} is that the functor is not necessarily of finite type. Consequently, the representing object can be unbounded.

\begin{lemma}
\label{lemma_epi}
Let $H\colon \bar\TT\to R\MMod$ be an admissible contravariant functor. Then there exists an object $A\in \TT^+$ and a morphism
$$\alpha\colon h^A\to H$$
of functors on $\bar \TT$ such that $\alpha$ is surjective on $\bar\TT_0$.
\end{lemma}
\begin{proof}
By admissibility of $H$ for any $k$ the $R$-module $H(G[k])$ is finitely generated.  Choose  a generating set $e_{k,1}, \ldots, e_{k,d_k}$  for the $R$-module $H(G[k])$.
Let
$$A=\oplus_k G^{\oplus d_k}[k].$$
Since $H(G[k])=0$ for $k\gg 0$, we have $A\in \TT^+$ (even $A \in \TT ^+_0$).
By Yoneda Lemma, one has $\Hom(h^A,H)=H(A)$ where $\Hom$ denotes morphisms of functors on $\bar\TT$. By admissibility of $H$ we have
$$H(A)=\prod H(G[k])^{\oplus d_k}.$$
Thus we have a distinguished element $\xi$ in $H(A)$, corresponding to the family $(e_{k,j})\in \prod H(G[k])^{\oplus d_k}$. Let $\alpha\colon h^A\to H$ be the corresponding morphism of functors on $\bar\TT$.
We claim that $\alpha$ is surjective on any object $G[k]$. 
Indeed, let $s_{k,j}$ denote the composition of the embedding 
$G[k]\to G[k]^{\oplus d_k}$ of the $j$-th summand and the embedding $G[k]^{\oplus d_k}\to \oplus_i G[i]^{\oplus d_i}$. We have a commutative diagram
$$\xymatrix{h^A(A)\ar[rr]^{\alpha(A)} \ar[d]^{h^A(s_{k,j})} && H(A) \ar[d]^{H(s_{k,j})} \\
h^A(G[k]) \ar[rr]^{\alpha(G[k])} && H(G[k]),
}$$
where $\alpha(A)(1_A)=\xi$ by Yoneda Lemma and  $H(s_{k,j})(\xi)= e_{k,j}$ by the definition of $\xi$. Therefore, the image of the map $\alpha(G[k])$ contains all elements $e_{k,j}$ for $j=1,\ldots, d_k$. Thus $\alpha(G[k])$ is surjective.

Since both $h^A$ and $H$ take direct sums to direct products,  $\alpha$ is also surjective on~$\bar \TT_0$.
\end{proof}

\begin{lemma}
\label{lemma_anni}
Let $H\colon \bar\TT\to R\MMod$ be an admissible cohomological functor. Let $A\in \TT^+$ be an object and
$\alpha\colon h^A\to H$ be a morphism
of functors on $\bar \TT$ such that $\alpha$ is surjective on $\bar\TT_0$. Then there exists an object $A'\in \TT^+$, a morphism $a\colon A\to A'$ and a commutative diagram of functors on $\bar\TT$
$$\xymatrix{h^A\ar[rr]^{a_*}\ar[rd]_{\alpha} && h^{A'} \ar[ld]^{\alpha'}\\ & H &}$$
such that $\alpha'$ is surjective on $\bar \TT_0$ and $a_*\colon h^A\to h^{A'}$ annihilates $\ker \alpha$ on $\bar\TT_0$.
\end{lemma}
\begin{proof}
Consider the functor $H'=\ker \alpha\colon \bar\TT\to R\MMod$. Since $h^A$ and $H$ are
admissible, $H'$ is also such by Lemma \ref{lemma_kernelofadmissible}. By Lemma~\ref{lemma_epi}, one can choose an object $B\in \TT^+$ and a morphism of functors $h^B\to H'$ on $\bar\TT$ which is surjective on $\bar\TT_0$. Consider the composite map $h^B\to h^A$ of functors on $\bar\TT$. Since $A,B\in \TT^+\subset \bar\TT$ the morphism of functors $h^B\to h^A$ is induced by a map $b\colon B\to A$. Let $A'=Cone(b)$ and $a$ be the natural morphism $A\to A'$. Clearly, $A'\in \TT^+$. By Yoneda Lemma, the sequence
$$\Hom(h^B,H)\gets \Hom(h^A,H)\gets \Hom(h^{A'},H)$$
is isomorphic to
\begin{equation}
\label{eq_BAA}
H(B)\xleftarrow{H(b)} H(A)\xleftarrow{H(a)} H(A').
\end{equation}
Sequence (\ref{eq_BAA}) is exact because $H$ is cohomological.
Obviously, the composition $h^B\xra{b_*} h^A\xra{\alpha} H$ is zero, i.e. $H(b)(\alpha)=0$. Since sequence  (\ref{eq_BAA}) is exact, there is $\alpha'\in H(A')$ such that $H(a)(\alpha')=\alpha$, which means $\alpha'\circ a_*=\alpha$.
$$\xymatrix{
h^B\ar[rr]^{b_*}\ar@{->>}[d] && h^A\ar[rr]^{a_*}\ar@{->>}[d]^{\alpha} && h^{A'}\ar@{-->>}[lld]^{\alpha'}\\
H'\ar@{^{(}->}[rru] && H &&
}$$
Surjectivity of $\alpha$ on $\bar{\TT}_0$ implies the surjectivity of $\alpha'$ on $\bar{\TT }_0$. Also the morphism
$a_*\colon h^A\to h^{A'}$ annihilates $\ker \alpha$ on $\bar\TT_0$ by construction.
\end{proof}

\begin{proof}[of Proposition~{\ref{prop_Brown}}]
Using Lemmas~\ref{lemma_epi} and~\ref{lemma_anni}, one can construct a sequence
$$A_1\xra{a_1} A_2\xra{a_2} A_3\to\ldots$$
in $\TT^+$ and a commutative diagram of functors $\bar\TT\to R\MMod$
$$\xymatrix{h^{A_1}\ar[rr]^{a_{1*}}\ar[rd]^{\alpha_1} && h^{A_2} \ar[rr]^{a_{2*}}\ar[ld]_{\alpha_2}&& h^{A_3}\ar[rr]^{a_{3*}}\ar[llld]^{\alpha_3} &&\ldots \\ & H &&&&&}$$
such that on the category $\bar\TT_0$ all morphisms $\alpha_i$ are surjective and any
$a_{i*}$ annihilates $\ker \alpha_i$. It follows then from \cite[Proposition 2.3.4]{BVdB} that $H$ is a direct summand of the functor $h^{A_{2k+2}}$ on the category $\bar\TT_k$.
\end{proof}

Recall that the category $\TT$ is regular, that is $\TT=\TT_n$ for some $n$.

\begin{lemma}
\label{lemma_TT}  
If $\TT=\TT_n$ then 
$\TT^+\subset \bar\TT_m$ for $m=2n+1$.
\end{lemma}
\begin{proof}
Recall that $\TT^+=\cup_k \TT^+_k=\cup_k \langle G\rangle^{\oplus,D^+_{\coh} (X)}_k$ and $\langle G\rangle^{\oplus,D^+_{\coh}(X)}_k$ is the idempotent closure of $[G]^{\oplus,D^+_{\coh }(X)}_k$ (in $D^+_{\coh} (X)$ or in $D(\qcoh X)$).
Since $\bar\TT_{2n+1}$ is idempotent closed in $D(\qcoh X)$, it suffices to check that $[G]^{\oplus,D^+_{\coh}(X)}_k\subset \bar\TT_{2n+1}$ for every $k$.

Let $L$ be an object in $[G]^{\oplus,D^+_{\coh}(X)}_k$. We claim that there exists a
commutative diagram 
\begin{equation}
\label{eq_LLL}
\xymatrix{L_1\ar[r] \ar[d] & L_2\ar[r]\ar[ld] & L_3\ar[r]\ar[lld]&  \ldots, \\
L &&&}
\end{equation}
where $L_i\in \TT$ and for any $r\in \Z$ one has $Cone(L_i\to L)\in D^{\ge r}_{\coh}(X)$ for $i\gg 0$. Below we will denote such diagrams by $L_1\to L_2\to \ldots \to L$.

We prove this claim by induction in $k$.  Let $k=0$, then $L$ is a direct sum of shifts of $G$. Since $L\in \TT^+\subset D^+_{\coh}(X)$, this sum is at most countable:
$L\simeq\oplus_{i=1}^{\infty} G[d_i]$, and $d_i\to -\infty$ as $i\to +\infty$. One can take  $L_s=\oplus_{i=1}^{s} G[d_i]$.

Now suppose the claim is known for some $k$. Let $L\in [G]^{\oplus,D^+_{\coh} (X)}_{k+1}$. By the definition, $L$ is quasi-isomorphic to a cone of some morphism
$L'\to L''$ where $L'\in [G]^{\oplus,D^+_{\coh }(X)}_{0}$ and $L''\in [G]^{\oplus,D^+_{\coh} (X)}_{k}$.  There exist
diagrams $L'_1\to L'_2\to \ldots \to L'$ and $L''_1\to L''_2\to \ldots \to L''$ 
with $L'_i,L''_i\in\TT$ provided by the induction hypothesis.
Note that $L'_s\in D^b_{\coh}(X)$ and for any $r\in\Z$ we have $Cone(L''_i\to L'')\in D^{\ge r}_{\coh}(X)$ for $i\gg 0$. Therefore for fixed $s$ and $i\gg 0$ one has an isomorphism
$$\Hom(L'_s,L''_i)\simeq \Hom(L'_s,L'').$$
Hence, the composition $L'_s\to L'\to L''$ factors through  $L''_i$ for $i\gg 0$.
Replacing $L''_i$ by a subsequence, we can obtain a commutative diagram
$$\xymatrix{L'_1\ar[r] \ar[d] & L'_2 \ar[r]\ar[d] & \ldots \ar[r] & L'\ar[d] \\
L''_1\ar[r]  & L''_2 \ar[r] & \ldots \ar[r] & L'',
}$$
whose rows satisfy the requirements of the claim. Now take $L_s:=Cone(L'_s\to L''_s)\in \TT$. 
There exist morphisms $L_s\to L$ completing the diagram
$$\xymatrix{
L'_s\ar[r]\ar[d] & L''_s\ar[r]\ar[d] & L_s\ar[r]\ar@{-->}[d] & L'_s[1]\ar[d] \\
L'\ar[r] & L''\ar[r] & L\ar[r] & L'[1].
}$$
Since for any $r\in \Z$ one has $Cone(L'_i\to L'), Cone(L''_i\to L'')\in D^{\ge r}_{\coh}(X)$ for $i\gg 0$, we get that $Cone(L_i\to L)\in D^{\ge r}_{\coh}(X)$ for $i\gg 0$.
Again, passing to a subsequence in $L_i$ we may assume that $\Hom(L_i,L_{i+1})=\Hom(L_i,L)$ for any $i$. Thus there exists a commutative diagram
(\ref{eq_LLL})
with the required properties, the claim is proven.

To finish the proof of Lemma \ref{lemma_TT} we use the construction of the homotopy colimit (see \cite{BN}), which is recalled below.  Consider the
complex
$$\hocolim L_i:=Cone(\oplus_i L_i\xra{1-\sigma} \oplus_i L_i),$$
where $\sigma$ denotes the collection of embeddings $L_i\to L_{i+1}$. 
Consider the morphism $\oplus_i L_i\to L$ defined by morphisms $L_i\to L$ from above.
Clearly, its composition with $1-\sigma$ is zero, hence it induces some  morphism $\delta \colon \hocolim L_i\to L$. We claim that $\delta$ is a quasi-isomorphism. Indeed, for each $p\in\Z$ we have a commutative diagram
$$\xymatrix{
\oplus_i H^p(L_i)\ar[r]^{1-\sigma} & \oplus_i H^p(L_i)\ar[r]\ar[rd]& H^p(\hocolim L_i)\ar[r]
\ar[d]^{H^p(\delta)} & 
\oplus_i H^{p+1}(L_i)\ar[r]^{1-\sigma} & \oplus_i H^{p+1}(L_i).\\
&& H^p(L) &&
}$$ 
Here the maps $1-\sigma$ in the top row are injective and the maps $H^p(L_i)\to H^p(L)$ are isomorphisms for $i\gg 0$. Therefore, the map $H^p(\delta)$ is also an isomorphism for any $p$.
It follows that the map 
$\delta \colon \hocolim L_i\to L$ is a quasi-isomorphism.

By assumption, $L_i\in \TT=\TT_n\subset \bar \TT_n$ for any $i$. Since $\bar \TT_n$ is closed under all direct sums, we get $\oplus_i L_i\in \bar\TT_n$. Finally, by Lemma~\ref{lemma_generation}, we have $L\simeq \hocolim L_i\in \bar\TT_{2n+1}$.
\end{proof}

Now we are ready for the
\begin{proof}[of Proposition~{\ref{prop_adjoint}}]
Let $M\in D^+_{\coh}(X)$ be an object. We need to demonstrate that the functor $h^M$ on $\TT^+$ is representable by an object of $\TT^+$.
By Lemma~\ref{lemma_TT} one has $\TT^+\subset \bar \TT_m$ for some $m$.
The functor $h^M\colon \bar {\TT}\to R\MMod$ is cohomological and admissible by Lemma~\ref{lemma_kernelofadmissible}.1.
Hence, by Proposition~\ref{prop_Brown} the restriction of the functor $h^M$ to the subcategory $\bar\TT_m$ is a direct summand of a representable functor $h^{C}$ for some $C\in \TT^+\subset \bar\TT_m$.
Consider the projector $h^C\to h^C$ onto this summand  as a morphism of functors on $\bar\TT_m$. By Yoneda Lemma it comes from a projector $\pi\colon  C\to C$. Since the category $\TT^+$ is Karoubian, projector $\pi$ splits, let $C_M\in \TT^+$ be the image of $\pi$. Clearly, $h^M\simeq h^{C_M}$ as functors on $\bar\TT_m$ and hence also as functors on~$\TT^+$. This proves that $h^M$ is representable on $\TT^+$. 

A standard argument proves (using Yoneda Lemma) that the correspondence $M\mapsto C_M$ is a functor and that there exists a functorial isomorphism $\Hom(N,M)\simeq \Hom(N,C_M)$ for $N\in \TT^+, M\in D^+_{\coh}(X)$. We leave this to the reader.
\end{proof}

\bigskip

\begin{proof}[of Theorem~{\ref{theorem_regular}}]
Let $G\in \TT$ be a generator, suppose $G\ne 0$. 
By Proposition~\ref{prop_adjoint}, for any $M\in D^b(\coh X)$ there exists $C_M\in \TT^+$ and a morphism $\eta_M\colon C_M\to M$ (adjunction counit) such that $\eta_M$ induces an
isomorphism
$$\Hom(N,C_M)\to \Hom(N,M)$$
for any $N\in \TT$. Let $E_M$ be the cone of $\eta_M$. Then clearly $E_M$ is right orthogonal to $\TT$ and hence
\begin{equation}
\label{eq_EG}
E_M\in G^{\perp}\cap  D^+_{\coh}(X).
\end{equation}

First, we check that the support $\Supp G$ of $G$ equals to all $X$. Take $M=\OO_X$ and consider the distinguished triangle $C_{\OO_X}\to \OO_X\to E_{\OO_X}$. It provides an exact sequence $H^0(C_{\OO_X})\to \OO_X\to H^0(E_{\OO_X})$.
One has $$X=\Supp\OO_X\subset \Supp H^0(C_{\OO_X})\cup \Supp H^0(E_{\OO_X}).$$
Since $C_{\OO_X}\in \TT^+$, we have $\Supp H^0(C_{\OO_X})\subset \Supp G$. Consequently,
$$X=\Supp G\cup \Supp H^0(E_{\OO_X}).$$
Hence $X$ is covered by two closed subsets.
Suppose that $\Supp G$ is not equal to $X$. Then $\Supp H^0(E_{\OO_X})\ne\emptyset$. Since $R$ is indecomposable, $X$ is connected. It follows that $\Supp G\cap \Supp H^0(E_{\OO_X})\ne\emptyset$. Recall that $E_{\OO_X}\in D^+_{\coh}(X)$.
Therefore by Lemma~\ref{lemma_supp} one has $\Hom^{\bul}(G,E_{\OO_X})\ne 0$, 
we get a contradiction with (\ref{eq_EG}).

Second, we claim that $D^b(\coh X)\subset \TT^+$. Take any $M\in D^b(\coh X)$. 
Then by (\ref{eq_EG}) and Corollary~\ref{cor_support} we get 
$E_M\in G^{\perp}\cap D^+_{\coh}(X)=0$.
Thus $M\simeq C_M\in \TT^+$. We proved that $D^b(\coh X)\subset \TT ^+$.

Finally, we argue that $D^b(\coh X)\subset \TT$.
By Lemma \ref{lemma-d^b-compact}, $D^b(\coh X)\subset D^+(\qcoh X)^c$, hence also $D^b(\coh X)\subset D^+_{\coh }(X)^c$.
In particular the object $G$ is compact in $D^+_{\coh}(X)$.
By Lemma~\ref{lemma_neeman} applied to the compact object $G$ in $D^+_{\coh}(X)$, we have
$$\TT^+\cap (D^+_{\coh} (X))^c=\TT.$$
Therefore
$$D^b(\coh X)=D^b(\coh X)\cap \TT^+\subset D^+_{\coh}(X)^c\cap \TT^+=\TT$$
which completes the proof of
Theorem~\ref{theorem_regular}.
\end{proof}

Theorem~\ref{theorem_regular} has a natural generalization to disconnected affine schemes.
\begin{corollary}
\label{corollary_regular}
Let $R_1,\ldots, R_n$ be commutative Noetherian indecomposable rings.
Let $R=\prod_{i=1}^n R_i$. Then any thick regular subcategory $\TT\subset D^b(R\mmod)$ has the form
$$\prod_{i\in I} D^b(R_i\mmod)$$
for some $I\subset \{1,\ldots,n\}$.
\end{corollary}
\begin{proof}
Suppose $\TT=\langle G\rangle_k$ is a thick regular subcategory in $D^b(R\mmod)$.  Note that
$$D^b(R\mmod)\simeq \prod_{i=1}^n D^b(R_i\mmod).$$
An object $M=(M_1,\ldots,M_n)\in D^b(R\mmod)$ is isomorphic to the direct sum
$$\oplus_{i=1}^n (0,\ldots,M_i,\ldots,0).$$
If $M\in \TT$ then also $(0,\ldots,M_i,\ldots,0)\in \TT$ for every $i$ (because $\TT$ is thick)
and we get
$$\TT=\prod_{i=1}^n\TT_i$$
for some thick subcategories $\TT_i\subset D^b(R_i\mmod)$. Let $G=(G_1,\ldots,G_n)$.
One can easily check that $\TT_i=\langle G_i\rangle_k$, hence any $\TT_i$ is regular. By Theorem~\ref{theorem_regular}, for any $i$ one has either $\TT_i=0$ or $\TT_i=D^b(R_i\mmod)$. The statement follows.
\end{proof}

Using Theorem~\ref{theorem_regular} one can study fully faithful functors between derived categories of affine schemes. The result is that there are no non-trivial ones. 

\begin{theorem}
\label{theorem_appl}
Let $X\neq \emptyset$ and $Y$ be two Noetherian affine schemes. Assume that $Y$ is connected and that the category $D^b(\coh X)$ is regular (for example, $X$ can be a scheme of finite type over a perfect field).
Then any fully faithful functor $D^b(\coh X)\to D^b(\coh Y)$ is an equivalence.
\end{theorem}

\begin{proof} Indeed, by Proposition \ref{prop_karoubian} the category $D^b(\coh X)$ is Karoubian, hence its essential image in $D^b(\coh Y)$ is a thick subcategory. Now the assertion is an immediate consequence of Theorem \ref{theorem_regular}.
\end{proof}

We came across the following question in the process of working on Theorem \ref{theorem_regular}. We still do not know the answer.

\medskip

\noindent{\bf Question.} Is there an example of a Noetherian integral domain $R$ and a nonzero full triangulated subcategory $\TT \subset D^b(R\mmod)$ with the following property: there exists a nonzero $f\in R$ such that for any object $E\in \TT$ the cohomology $H^\bullet (E)$ is annihilated by~$f$?

\begin{remark} 
One can demonstrate that the analogous question about a triangulated subcategory $\TT \subset \Perf(R)$ has the negative answer. Indeed, passing to the thick closure one can assume such subcategory $\TT$ to be thick.
Recall the Hopkins-Neeman bijection (see the proof of Theorem~\ref{theorem_regular0}) between thick subcategories in $\Perf (R)$ and closed under specialization subsets in $\Spec R$: the subset $Z\subset \Spec R$ corresponds to the thick subcategory $\TT_Z=\{C\in \Perf(R)\mid \Supp C\subset Z\}$ in $\Perf (R)$. Hence, it suffices to consider $\TT=\TT_{Z_0}$ for some  non-empty subset  $Z_0\subset \Spec R$, closed under specialization. Suppose that for some $f\in R$ any cohomology of any object in $\TT_{Z_0}$ is annihilated by $f$. Take some closed point $z\in Z_0$, let $Z=\{z\}$ and $\II_Z$ be the ideal of $Z$. Applying the arguments from the first proof of Theorem \ref{theorem_regular0} to this $Z$, we get that $f\in \cap_i \II_Z^i=0$.  

As far as we know there is no classification of thick subcategories in $D^b(R\mmod)$. An answer to the above question would be a step in this direction.
\end{remark}

\section{$f$-approximations}
\label{section_f}

Recall that $R$ denotes a commutative Noetherian ring and $X=\Spec R$.

In Section \ref{section_regsubcat} we studied thick regular subcategories of $D^b(R\mmod)$. 
In this section we consider a more general situation: a regular triangulated category $\TT$ and a functor (not necessarily full or faithful) $\TT\to D^b(R\mmod)$. We prove that if the image of this functor contains a collection of objects with certain properties, then the Rouquier dimension of $\TT$ cannot be smaller than a certain number. An example of such situation is described in Corollary \ref{cor_bounddim}. We then apply this result to the study of point-like objects in regular categories in Section~\ref{section_pointlike}. The proof of the main Theorem \ref{theorem_main} is based on the notion of an $f$-approximation (specifically on Lemma \ref{lemma_main}) combined with the well-known Ghost Lemma. As a byproduct of these methods we also obtain an application to cohomological annihilator of $R$, see Theorem \ref{theorem_ca}.

\begin{definition} 
\label{definition_fapprox}
Let $\TT$ be an $R$-linear triangulated category. Consider objects $G,M\in \TT$ and an element $f\in R$.
We say that $M$ can be \emph{$f$-approximated by $G$ in $n$ steps}, denoted $M\in_f [G]_n$, if there exists an object $P\in [G]_n$ and a morphism $\phi\colon P\to M$ such that $Cone(\phi )$ is annihilated by $f$.

We say that $M$ can be \emph{classically $f$-approximated by $G$ in $n$ steps}, denoted $M\in_f \langle G\rangle_n$, if there
exists $M'\in \TT$ such that $M\oplus M'\in_f [G]_n$.
\end{definition}

The next lemma is well known, see, for example, \cite[Theorem 3]{Ke} or \cite[Lemma 4.11]{Rou}.
\begin{lemma}[(Ghost Lemma)]
\label{lemma_ghost}
Let $F_0\xra{g_1}F_1\xra{g_2} F_2\to \ldots \xra{g_{n+1}} F_{n+1}$ be morphisms in a triangulated category $\TT$, denote $g=g_{n+1}\ldots g_2g_1$. Suppose that $G\in\TT$ is an object such that the map $\Hom^i(G,F_{k-1})\xra{(g_k)_*}\Hom^i(G,F_k)$ vanishes for any $i$ and $k=1,\ldots ,n+1$. Then the map $\Hom^i(Q,F_{0})\xra{g_{*}}\Hom^i(Q,F_{n+1})$ vanishes for any $i$ and $Q\in \langle G\rangle_n$.
\end{lemma}

A proof of the following lemma can be found, for example, in~\cite[Th. 2.10]{IT}. We include a proof for the convenience of the reader.

\begin{lemma}
\label{lemma_sss}
Let $R$ be a commutative Noetherian ring and $\TT=D(R\MMod)$. Fix an element $f\in R$ and let $M\in_f \langle R\rangle_n$ for some $n$.
\begin{enumerate}
\item Let $S_0,\ldots, S_{n+1}$ be $R$-modules and let $\theta_k\colon S_{k-1}\to S_k[1]$ be morphisms in $\TT$ for $k=1,\ldots ,n+1$. Then the image of the map
$$\Hom(M,S_0)\xra{(\theta_{n+1}\ldots \theta_1)_*} \Hom(M,S_{n+1}[n+1])$$ is annihilated by $f$.

\item Suppose moreover that $M$ is just an $R$-module. Then for any $R$-module $N$ and $m>n$ one has
$$f\cdot \Ext^m(M,N)=0.$$
\end{enumerate}
\end{lemma}

\begin{proof}
1. Denote the composition $\theta_{n+1}\ldots \theta_1$ by $\theta$. By definition, there exists an object $M'\in D^b(R\MMod)$ and a distinguished triangle $P\to M\oplus M'\to C\to P[1]$ such that $f\cdot C=0$ and $P\in [R]_n$. Note that the map 
\begin{equation}
\label{eq_zeromap}
\Hom^i(R,S_k[k])\to \Hom^i(R,S_{k+1}[k+1]),
\end{equation} 
induced by $\theta_k$ is zero for any $i$. Indeed, either the source or the target of (\ref{eq_zeromap}) is a zero module because $S_k$ and $S_{k+1}$ are $R$-modules. Applying 
 Ghost Lemma, we get that the map $\Hom(P,S_0)\to \Hom(P,S_{n+1}[n+1])$ induced by $\theta$ is zero. Consider the commutative diagram with exact rows where the vertical maps are induced by $\theta$:
$$\xymatrix{\Hom(P,S_0)\ar[d]^0 & \Hom(M\oplus M',S_0)\ar[d]\ar[l] & \Hom(C,S_0)\ar[d]
 \ar[l] \\
 \Hom(P,S_{n+1}[n+1]) & \Hom(M\oplus M',S_{n+1}[n+1])\ar[l] &\Hom(C,S_{n+1}[n+1]). \ar[l].
}$$
Since $f\cdot C=0$, a diagram chase shows that the image of the middle vertical arrow is annihilated by $f$.
Clearly, it follows that the image of $\Hom(M,S_0)\to \Hom(M,S_{n+1}[n+1])$ is annihilated by $f$.

2. Recall that any element $\xi\in \Ext^m(M,N)=\Hom(M,N[m])$ can be represented as the composition of some maps $M=S_0\xra{\theta_1} S_1[1]\xra{\theta_2} S_2[2]\to\ldots \xra{\theta_m} S_{m}[m]=N[m]$ for some $R$-modules~$S_i$. Clearly, the composition $\theta_{n+1}\ldots\theta_1$ is the image of $1_M$ under the map $\Hom(M,S_0)\to \Hom(M,S_{n+1}[n+1])$ and by part 1 is annihilated by $f$. Consequently, $\xi=\theta_m\ldots\theta_{n+1}\ldots\theta_1$ is also annihilated by $f$ for $m>n$.
\end{proof}

The following statement is standard, it will be used frequently in the forthcoming arguments. We recall the proof for the benefit of the reader.
\begin{lemma}
\label{lemma_ABC}
Let $\TT$ be a triangulated category linear over a commutative ring $R$. Let $A\to B\to C\to A[1]$ be a distinguished triangle in $\TT$. Suppose $A$ is annihilated by $f_A$ and 
$B$ is annihilated by $f_B$, where $f_A,f_B\in R$. Then $C$ is annihilated  by $f_Af_B$.
\end{lemma}
\begin{proof}
Consider the commutative diagram
$$\xymatrix{
B \ar[r]^b \ar[d]_{f_A} & C \ar[r]^c \ar[d]^{f_A} \ar@{-->}[ld]^{b'} & A[1] \ar[d]^{f_A} \\
B \ar[r]_b \ar[d]_{f_B} & C \ar[r]^c \ar[d]^{f_B} & A[1] \ar[d]^{f_B} \ar@{-->}[ld]^{c'} \\
B \ar[r]^b              & C \ar[r]_c              & A[1]. 
}$$
By the properties of a triangulated category, since $cf_A=f_Ac=0\in\Hom(C,A[1])$ there is $b'\in \Hom(C,B)$ such that $bb'=f_A$.  Similarly, there is $c'\in \Hom(A[1],C)$ such that $c'c=f_B$. It follows that 
$$f_Bf_A=c'cbb'=0$$
as morphisms from $C$ to $C$.
\end{proof}

\begin{lemma}
\label{lemma_main}
Let $R$ be a commutative Noetherian integral domain and $G\in D^b(R\mmod)$ be an object.
Then there exists a non-zero element $f\in R$ and a function $g\colon \Z_{\ge 0} \to \Z_{\ge 0}$ such that for any $s\ge 0$ and any object $M_s\in [G]_s$ one has $M_s\in_{f^{g(s)}}[R]_s$. In other words, for any $s\ge 0$ and any object $M_s\in [G]_s$ there exists an object $P_s\in [R]_s$ and a morphism $\phi_s \colon P_s\to M_s$ such that $Cone(\phi _s)$ is annihilated by $f^{g(s)}$.
\end{lemma}

\begin{remark}
In fact, our proof shows that one can take $g(s)=2^{s+1}-1$.
\end{remark}

\begin{proof}
Choose a graded module $P$ that is a direct sum of shifted free $R$-modules of finite rank and an injective morphism $\alpha \colon P\to H^\bullet (G)$,
such that $\coker (\alpha)$ is a torsion $R$-module, i.e. $\ann(\coker(\alpha))\neq 0$.
(Such $\alpha $ exists since $R$ is a domain).
We consider $P$ as a complex of $R$-modules with zero differential, and choose a morphism $\tilde{\alpha}\colon P\to G$
of complexes, such that $\alpha =H^\bullet (\tilde{\alpha})$. Then any cohomology of $Cone(\widetilde\alpha)$ is annihilated by the (nonzero) ideal $\ann(\coker(\alpha))$. Since $Cone(\widetilde\alpha)$ is bounded it follows from Lemma~\ref{lemma_ABC} that $Cone(\widetilde\alpha)$ itself is annihilated (as an object of $D^b(R\mmod)$) by some power
$\ann(\coker(\alpha))^t$ of this ideal. Choose any nonzero element $f\in \ann(\coker(\alpha))^t$. We claim that there exists a function $g\colon \Z_{\ge 0} \to \Z_{\ge 0}$ such that $f$ and $g$ have the desired properties.

We proceed by induction on $s$. For $s=0$ any object $M_0\in [G]_0$ is a finite direct sum of shifts of $G$, hence we can take for $\phi _0$ a direct sum of shifts of the morphism $\tilde{\alpha }\colon P\to G$ chosen above (clearly $P\in [R]_0$). Then the cone of $\phi_0$ is a direct sum of shifts of $Cone(\widetilde\alpha)$ and thus  annihilated by $f$.
We put $g(0)=1$.

Assume that we proved the assertion for $s-1$.
Let $M_s\in [G]_s$ be the cone of a morphism $\psi \colon M_0\to M_{s-1}$ with $M_0\in [G]_0$ and $M_{s-1}\in [G]_{s-1}$.
Consider the diagram of morphisms of complexes (denoted by solid arrows)
$$\xymatrix{
P_{s-1} \ar[rr]^{\phi _{s-1}} && M_{s-1} \\
P_{0} \ar@{-->}[u]^{\beta} \ar[rr]^{ \phi _{0}} && M_{0}\ar[u]^{\psi}
}
$$
provided by the induction hypothesis.
If there existed a morphism $\beta \colon P_0\to P_{s-1}$ which complemented the above diagram to a commutative square,
then there would exist a morphism $\gamma \colon Cone(\beta)\to M_s=Cone(\psi)$ satisfying the requirement of the lemma for an appropriate value $g(s)\in \Z_{\ge 0}$.

We claim that there exists a morphism $\beta \colon P_0\to P_{s-1}$ which completes the above diagram with $\phi_0$ replaced by $f^{g(s-1)}\phi_0$. Indeed, consider the exact sequence
$$\Hom(P_0,P_{s-1})\to \Hom(P_0,M_{s-1})\stackrel{\delta}{\to }\Hom(P_0,Cone(\phi_{s-1}))$$
The element $f^{g(s-1)}\psi\phi_0\in \Hom(P_0,M_{s-1})$ maps to zero by $\delta$ because $\delta (\psi\phi_0)\in \Hom(P_0,Cone(\phi_{s-1}))$ and $Cone(\phi_{s-1})$ is annihilated by $f^{g(s-1)}$ by the induction hypothesis. Therefore there exists $\beta\colon P_0\to P_{s-1}$ such that $\phi_{s-1}\beta=f^{g(s-1)}\psi\phi_0$.
We obtain the commutative diagram
$$\xymatrix{
P_0[1] && M_0[1] && C_{0,s}[1] && \\
P_{s} \ar[rr]^{\phi_s}\ar[u] && M_s \ar[rr]\ar[u] && C_s \ar[rr] \ar[u] && P_s[1] \\
P_{s-1} \ar[rr]^{\phi _{s-1}} \ar[u] && M_{s-1} \ar[u] \ar[rr] && C_{s-1}\ar[u]\ar[rr] && P_{s-1}[1]\\
P_{0} \ar[u]^{\beta} \ar[rr]^{f^{g(s-1)} \phi _{0}} && M_{0}\ar[u]^{\psi} \ar[rr]  && C_{0,s} \ar[u]\ar[rr] && P_0[1],
}
$$
where all rows and columns are distinguished triangles, see \cite[Prop. 1.1.11]{BBD}. Clearly, $P_s\in [R]_s$.
Let us check that $C_s$ is annihilated by a certain power of $f$. By induction hypothesis, $C_{s-1}$ is annihilated by $f^{g(s-1)}$. The map $f^{g(s-1)}\cdot \phi_0\colon P_0\to M_0$ is the composition
$$P_0\xra{f^{g(s-1)}} P_0\xra{\phi_0} M_0.$$
By the octahedron axiom, there exists a distinguished triangle
$$Cone(f^{g(s-1)})\to C_{0,s}\to Cone(\phi_0)\to Cone(f^{g(s-1)})[1].$$
Clearly the object $Cone(f^{g(s-1)})$ is annihilated by $f^{g(s-1)}$. Indeed, $P_0$ is a direct sum of shifted free $R$-modules of finite rank and $R$ is a domain; therefore the object $Cone(f^{g(s-1)})$ is isomorphic to the direct sum of its cohomologies which are annihilated by $f^{g(s-1)}$. Also $Cone(\phi_0)$ is annihilated by $f$ as pointed out in the beginning of the proof. Therefore
$f^{g(s-1)+1}\cdot C_{0,s}=0$ by Lemma~\ref{lemma_ABC}. Finally, Lemma~\ref{lemma_ABC} implies that $Cone(\phi_s)$ is annihilated by $f^{2g(s-1)+1}$, and we put $g(s):=2g(s-1)+1$.

This completes the induction step and proves the lemma.
\end{proof}

\begin{theorem}
\label{theorem_main}
 Let $R$ be a commutative Noetherian integral domain, let $G\in  D^b(R\mmod)$ be an object and $\{M_i\}_{i\in I}$ be a family of finitely generated $R$-modules such that $M_i\in\langle G \rangle_n$ for some $n$ and any $i$. Suppose there exist an integer $d\ge 1$ and a family of $R$-modules $\{N_i\}_{i\in I}$ such that $\cap_i \ann\Ext^d(M_i,N_i)=0$.   Then $n\ge d$.
\end{theorem}
\begin{proof} By definition for any module $M_k$ there exists some object $M'_k\in D^b(R\mmod)$ such that $M_k\oplus M'_k\in [G]_n$. We apply Lemma~\ref{lemma_main}: there exists a non-zero element $f$, such that $M_k\oplus M'_k\in_{f} [R]_n$ for any~$k$. It follows that $M_k\in_f \langle R\rangle_n$ for any $k$.  By Lemma~\ref{lemma_sss} we get that $f\cdot \Ext^m(M_k,N_k)=0$ for any $k$ and any $m>n$. In particular, $\cap_k \ann \Ext^m(M_k,N_k)\ne 0$. By our assumptions, we deduce that $d\le n$.
\end{proof}

\begin{corollary}
\label{cor_bounddim} 
Let $R$ be a commutative Noetherian integral domain. Let $\TT$ be a regular triangulated category with $\Rdim \TT =n$, and $F\colon \TT \to D^b(R\mmod)$ be a triangulated functor such that the essential image of $F$ contains a family $\{M_i\}_{i\in I}$ of finitely generated $R$-modules. Suppose there exist an integer $d\ge 1$ and a family of $R$-modules $\{N_i\}_{i\in I}$ such that $\cap_i \ann\Ext^d(M_i,N_i)=0$.  Then $d\le n$.
\end{corollary}

\begin{proof} Let $E\in \TT$ be an object such that $\TT =\langle E\rangle _n$. Denote the object $F(E)\in D^b(R\mmod)$ by $G$. Since $\TT =\langle E\rangle _n$ it follows that every object in the image of the functor $F$ belongs to the subcategory $\langle G\rangle _n\in D^b(R\mmod)$. In particular $M_i\in \langle G\rangle _n$. Now the assertion follows from Theorem \ref{theorem_main}.
\end{proof}

Proposition \ref{lemma_exafatext} gives an example where the assumptions of Theorem \ref{theorem_main} are fulfilled.
We start with a lemma.

\begin{lemma}
\label{lemma_matsumura}
Let $R$ be a local ring. Consider a regular sequence of elements $x_1,\ldots,x_n$ in the maximal ideal of $R$. Let $d_1,\ldots,d_n$ be positive integers. Then the sequence $x_1^{d_1},\ldots,x_n^{d_n}$ is also regular. Therefore if we put  $I_k:=(x_1^k,\ldots,x_n^k)$, then one has
$$\Ext^n(R/I_k,R)=R/I_k,\quad \Ext^i(R/I_k,R)=0\quad\text{for}\quad i\ne n.$$
\end{lemma}
\begin{proof} For the first statement see \cite[Thm 16.1]{Ma}. For the second one consider the Koszul complex of $R/I_k$.
\end{proof}

\begin{prop}
\label{lemma_exafatext}
Let $R$ be a commutative Noetherian integral domain and $\eta\in\Spec R$ be a point. Consider the corresponding local ring $R_\eta$ with the maximal ideal $\m$. Let $M_i, i\in I$ be a family of finitely generated $R$-modules such that among localizations $(M_i)_{\eta}$ there appear quotients $R_{\eta}/I$ for all  $\m$-primary ideals $I\subset R_\eta$. Let $d$ be the depth of the ring $R_{\eta}$. Then a condition from Theorem~\ref{theorem_main} is fulfilled (with $N_i=R$):
$$\cap_i \ann \Ext^d(M_i,R)=0.$$
\end{prop}
\begin{proof}
Choose a regular sequence $x_1,\ldots,x_d$ of elements in the maximal ideal $\m$ of $R_{\eta}$.
For any $k\ge 0$, let $I_k=(x_1^k,\ldots,x_d^k)\subset R_{\eta}$. By Lemma~\ref{lemma_matsumura},
$\Ext^d(R_{\eta}/I_k,R_{\eta})=R_{\eta}/I_k$. By assumptions, for some $i(k)\in I$ we have
$(M_{i(k)})_{\eta}\simeq R_{\eta}/I_k$ and
$$\Ext^d_R(M_{i(k)},R)_{\eta}=\Ext^d_{R_{\eta}}((M_{i(k)})_{\eta},R_{\eta})=\Ext^d_{R_{\eta}}(R_{\eta}/I_k,R_{\eta})=R_{\eta}/I_k.$$
Therefore,
\begin{multline*}
(\cap_i \ann \Ext^d(M_i,R))_{\eta}\subset \cap_k (\ann \Ext^d(M_{i(k)},R))_{\eta}=\cap_k \ann (\Ext^d(M_{i(k)},R)_{\eta})=\\
=\cap_k I_k\subset \cap_k \m^k=0,
\end{multline*}
where the last equality is due to Krull's Intersection Theorem.
Since $R$ is a domain, it follows that  $\cap_i \ann \Ext^d(M_i,R)=0$.
\end{proof}

As another application of the developed technique one can deduce non-triviality of the cohomological annihilator for rings $R$ such that the category $D^b(R\mmod)$ is regular.
Recall that the ideals $ca_k(R)\subset R$, $k\ge 0$
are defined by
$$ca_k(R)=\cap_{M,N}\ann \Ext^k_R(M,N),$$
where the intersection is taken over all finitely generated $R$-modules $M,N$.
Since any element in $\Ext^k_R(M,N)$ is the product of some elements $\xi\in \Ext^{k-1}(M,N')$ and $\xi'\in \Ext^1(N',N)$, one has 
$$ca_1(R)\subset ca_2(R)\subset \ldots \subset ca_k(R)\subset \ldots \subset ca(R),$$
where the ideal $ca(R)$ is defined   as $\cup_i ca_i(R)$.
\begin{theorem}
\label{theorem_ca}
Let $R$ be a commutative Noetherian integral domain such that the category $D^b(R\mmod)$ is regular (see Example~\ref{example_example}). Let $d=\Rdim(D^b(R\mmod))$. Then there exists a non-zero element $f\in R$ such that $\Ext_R^k(M,N)$ is annihilated by $f$ for any finitely generated $R$-modules $M$ and $N$ and any $k>d$. In particular, $ca_{d+1}(R)\ne 0$.
\end{theorem}
\begin{proof}
By our assumption there exists a generator $G\in D^b(R\mmod)$ such that $\langle G\rangle_d=D^b(R\mmod)$. By Lemma~\ref{lemma_main}, there exists a non-zero element $f\in R$ such that for any $R$-module $M$ one has $M\in_f \langle R\rangle_d$. Then by Lemma~\ref{lemma_sss}, we have $f\cdot \Ext^k_R(M,N)=0$ for any finitely generated $R$-modules $M,N$ and $k>d$.
\end{proof}

Of course, such results are  not new (for example, see \cite{IT2} where strong generation of abelian categories of modules is studied).
What seems to be interesting is the estimate: $ca_{d+1}(R)\ne 0$ where $d$ is the Rouquier dimension of $D^b(R\mmod)$.

\section{Point-like objects in regular categories}
\label{section_pointlike}

Fix a field $\kk$.

\begin{definition}
\label{definition_pointlike}
Let $\TT$ be an enhanced $\kk$-linear triangulated category. That is, there is given a pre-triangulated $\kk$-linear dg category $\AAA$ and an equivalence of triangulated categories $\TT \simeq [\AAA]$, where $[\AAA]$ is the homotopy category of $\AAA$, see \cite{BLL} for details. An object $E$ of $\TT$ is said to be
\emph{point-like} of dimension $r$ if the dg-algebra $\End_{\AAA}(E)$ is quasi-isomorphic to the exterior algebra
$\Lambda^{\bul}(V_r)$ of some $r$-dimensional $\kk$-vector space $V_r$. Here $\Lambda^{\bul}(V_r)$ is considered as a dg-algebra with zero differential and $\deg(V_r)=1$.
\end{definition}

Example: if $x\in X$ is a nonsingular point of a $\kk$-scheme $X$ with the residue field~$\kk$, then the skyscraper sheaf $\OO_x$ is a point-like object of $D^b(\coh X)$ of dimension equal to the dimension of $X$ at $x$, see~\cite[Prop. 5.2.]{ELO}.

Our goal in this section is to demonstrate that the dimension of point-like objects in a triangulated category is bounded above by the Rouquier dimension of the category.

Recall that a triangulated $\kk$-linear category $\TT$ is called \emph{$\Ext$-finite} if for any $E,F\in\TT$ one has $\dim_{\kk}\oplus_i\Hom(E,F[i])<\infty$.

\begin{theorem}\label{theorem_point}
Let $\TT$ be an enhanced $\kk$-linear $\Ext$-finite triangulated category. Suppose~$\TT$ has  finite Rouquier dimension $n$ and contains a point-like object $E$ of dimension $r$. Then $r\le n$.
\end{theorem}
\begin{proof} We may assume that the category $\TT$ is Karoubian.
Let $\AAA$ be the enhancement of $\TT$.
Denote by $\Lambda=\Lambda^{\bul}(V_r)$ the dg-algebra with zero differential and $\deg(V_r)=1$ which is quasi-isomorphic to $\End_{\AAA}(E)$.
Choose a quasi-isomorphism of dg algebras $\End_{\AAA}(E)\simeq \Lambda$. This induces an equivalence of the corresponding derived categories  $D(\End _{\AAA}(E))\simeq D(\Lambda)$. Let $D_{fg}(\Lambda)\subset D(\Lambda)$ be the full triangulated subcategory of dg modules with finite dimensional total cohomology.

We have the triangulated functor
$$\Phi\colon \TT\to D(\End _{\AAA}(E))\simeq D(\Lambda),\quad E\mapsto \Hom_{\AAA}(E,-).$$
Since $\TT$ is $\Ext$-finite, the functor $\Phi$ has its image in $D_{fg}(\Lambda)$.

Note that the functor $\Phi$ restricts to an equivalence between the subcategories $\langle E\rangle\subset \TT$ and $\Perf(\Lambda)\subset D_{fg}(\Lambda)$. Indeed, $\Phi(E)=\Lambda$,
and the map 
$\Phi \colon\Hom^\bullet (E,E)\to \Hom ^\bullet (\Lambda,\Lambda)$
is by definition an isomorphism. Hence $\Phi \vert _{\langle E\rangle}$ is full and faithful. Since both categories $\langle E\rangle$ and $\Perf (\Lambda)$ are Karoubian,
it is indeed an equivalence.

Consider $\kk$ as a right $\Lambda$-dg-module (where $\Lambda_{>0}$ acts by zero) and choose its h-projective resolution $P\to \kk$. Then $\Hom ^i _{D(\Lambda)}(\kk ,\kk)=H^i(\End (P))$.
By Lemma \ref{lemma_dgmod} below, one has $\Hom_{D(\Lambda)}^0(\kk,\kk)\simeq \kk[[t_1,...,t_r]]$ for $i=0$ and $\Hom_{D(\Lambda)}^i(\kk,\kk)=0$ otherwise.
Put
$$A=\kk[[t_1,...,t_r]].$$
This is a commutative Noetherian regular $\kk$-algebra. We have a quasi-isomorphism of dg algebras $\End (P)\simeq A$ and hence an equivalence of the corresponding derived categories $D(\End (P))\simeq D(A)$.
Consider the functor (Koszul duality)
$$K\colon D(\Lambda)\to D(\End (P))\simeq D(A)$$
induced by the dg-functor $\Hom(P,-)$ from the category of $\Lambda $-dg-modules to the category of $\End (P)$-dg-modules.

Recall that $D^b_{fg}(A)$ denotes the full triangulated subcategory of $D(A)$ consisting of complexes with finitely generated (over $A$) total cohomology. 
We claim that $K$ restricts to an equivalence between $D_{fg}(\Lambda)$ and $D^b_{fg}(A)$. 
Indeed, the category $D_{fg}(\Lambda)$ is classically generated by the module $\kk$ (see~\cite[Lemma 7.35]{Rou}).  Also, it follows from the definition that  $D_{fg}(\Lambda)$ is Karoubian.
Further, $K(\kk)=A$ and $\Hom^i_{D(\Lambda)}(\kk,\kk)=\Hom^i_{D(A)}(A,A)$ for all $i$.
It follows by standard argument that $K$ is fully faithful on $D_{fg}(\Lambda)$. It remains to note that $D^b_{fg}(A)$ coincides with $\Perf(A)$ and thus is classically generated by $A$. Also $D^b_{fg}(A)$ is Karoubian. Therefore
$$K\colon D_{fg}(\Lambda)\to D^b_{fg}(A)$$
is an equivalence.

We observe that $K(\Lambda)=\Hom^{\bul}(\kk,\Lambda)=\kk[-r]$. Indeed, $\Lambda$ is an h-injective $\Lambda$-dg-module, hence $$\Hom^{\bul}_{D(\Lambda)}(\kk,\Lambda)=H^{\bul}(\Hom_{\Lambda}(\kk,\Lambda))=\kk[-r].$$
It follows that $K$ restricts to an equivalence between the subcategory $\Perf(\Lambda)\subset D_{fg}(\Lambda)$ and the subcategory $D^b_{fg,0}(A)\subset D^b_{fg}(A)$ which consists of objects supported at the maximal ideal $\m \subset A$.

Consider the composite functor
$$\TT\xra{\Phi} D_{fg}(\Lambda)\xra{K} D^b_{fg}(A).$$
It restricts to an equivalence
$$K\Phi\colon \langle E\rangle \simeq  D^b_{fg,0}(A).$$
Let $G\in \TT$ be a generator such that $\TT=\langle G\rangle_n$.
In particular, for any object $C$ in $\langle E\rangle$ one has $C\in \langle G\rangle_n$.
Applying $K\Phi$, we get that any object in $D^b_{fg,0}(A)$ belongs to $\langle K\Phi(G)\rangle_n$. Therefore for any $\m$-primary ideal $I\subset A$ the $A$-module $A/I$ is in
$\langle K\Phi(G)\rangle_n$.
By Proposition~\ref{lemma_exafatext} we have
$$\bigcap_{I\,\text{is $\m$-primary}} \ann \Ext^r(A/I,A)=0.$$
By Theorem~\ref{theorem_main} we get that $n\ge r$.
\end{proof}

\begin{lemma} \label{lemma_dgmod} 
Let $V_r$ and $\Lambda=\Lambda^{\bul}(V_r)$ be as in Definition~\ref{definition_pointlike}.
Consider the right $\Lambda$-dg-module $\kk$ (where $\Lambda _{>0}$ acts by zero). Then
$$\Hom^0_{D(\Lambda)}(\kk,\kk)\simeq k[[t_1,...,t_r]],\qquad 
\Hom^i_{D(\Lambda)}(\kk,\kk)=0\quad\text{for}\quad i\neq 0,$$
where the first isomorphism is an isomorphism of algebras.
\end{lemma}

\begin{proof} 
Consider the free supercommutative graded algebra:
$$S^{\bul}(V_r\oplus V_r[-1])\simeq 
S^{\bul}V_r\otimes_{\kk} \Lambda\simeq \bigoplus_{i=0}^{\infty}S^{i}V_r\otimes_{\kk} \Lambda\simeq \kk[x_1,\ldots,x_d]\otimes_{\kk} \kk\langle \xi_1,\ldots,\xi_d\rangle 
,$$
where $x_1,\ldots,x_r$ is a basis for $V_r$, $\deg x_i=0$, 
$\xi_1,\ldots,\xi_r$ is a basis for $V_r[-1]$, $\deg \xi_i=1$. 
For $i=1,\ldots,r$ define $\Lambda$-linear operators $\partial_i\colon S^{\bul}V_r\otimes_{\kk} \Lambda\to S^{\bul}V_r\otimes_{\kk} \Lambda$ by 
$$\partial_i(x_1^{a_1}\cdot\ldots \cdot x_i^{a_i}\cdot\ldots\cdot x_r^{a_r})=\begin{cases}
x_1^{a_1}\cdot\ldots \cdot x_i^{a_i-1}\cdot\ldots\cdot x_r^{a_r}, & a_i>0;\\
0, & a_i=0.
\end{cases}$$
Define a
$\Lambda$-linear operator $d\colon S^{\bul}V_r\otimes_{\kk} \Lambda\to S^{\bul}V_r\otimes_{\kk} \Lambda$ by 
$$d=\sum_{i=1}^r \xi_i\partial_i.$$
This operator $d$ satisfies $d^2=0$, i.e. is a  differential, and makes $S^{\bul}V_r\otimes_{\kk} \Lambda$ a $\Lambda$-dg-module. (Note that $d$ is not a derivation of algebras!) As a $\Lambda$-dg-module, $S^{\bul}V_r\otimes_{\kk} \Lambda$ is free and quasi-isomorphic to $\kk$. Therefore,
\begin{multline*}
\Hom^{\bul}_{D(\Lambda)}(\kk,\kk)\simeq 
\Hom^{\bul}_{\Lambda}(S^{\bul}V_r\otimes \Lambda,S^{\bul}V_r\otimes \Lambda)= \Hom_{\Lambda}(S^{\bul}V_r\otimes \Lambda,S^{\bul}V_r\otimes \Lambda)\simeq\\
\simeq \Hom_{\Lambda}(S^{\bul}V_r\otimes \Lambda, \kk)\simeq 
\prod_{i\ge 0} S^iV_r^*\simeq 
\kk[[t_1,\ldots,t_r]],
\end{multline*}
where $t_1,\ldots,t_r\in V_r^*$ is the basis dual to $x_1,\ldots,x_r$.	

We aim to demonstrate that the isomorphism 
\begin{equation}
\label{eq_SS}
\End_{\Lambda}(S^{\bul}V_r\otimes \Lambda)\to \kk[[t_1,\ldots,t_r]]
\end{equation} 
constructed above is an isomorphism of algebras. Indeed, the maps $\partial_i\colon S^{\bul}V_r\otimes \Lambda\to S^{\bul}V_r\otimes \Lambda$ are linear over $\Lambda$, homogeneous of degree $0$ and commute with the differential~$d$. Hence $\partial_i\in \End_{\Lambda}(S^{\bul}V_r\otimes \Lambda)$. 
Further, $\partial_1,\ldots,\partial_r$ commute with each other and are locally nilpotent: for any element $y\in S^{\bul}V_r\otimes \Lambda$ there exists $k$ such that $\partial_i^k(y)=0$.
Therefore we get an algebra homomorphism 
\begin{equation}
\label{eq_SS2}
\kk[[t_1,\ldots,t_r]]\to \End_{\Lambda}(S^{\bul}V_r\otimes \Lambda)
\end{equation} 
which sends $t_i$ to $\partial_i$. 
Looking at the construction of (\ref{eq_SS}) one can check that (\ref{eq_SS}) is inverse to 
(\ref{eq_SS2}), this proves the claim.
\end{proof}


\end{document}